\documentclass[onecolumn,final]{elsart3p}
\usepackage{graphicx,color}
\usepackage{amsmath}       %   AMS equation environments
\usepackage{amssymb}       %   AMS symbol fonts (e.g., \boldsymbol{}
\usepackage{epstopdf}

\usepackage{bm,bbm}
\usepackage{overpic}
\usepackage{rotating}
\usepackage{comment}
\usepackage{changes}
\usepackage[colorinlistoftodos,prependcaption,textsize=tiny]{todonotes}

\theoremstyle{plain}
\newtheorem{theorem}{Theorem}[section]
\newtheorem{lemma}[theorem]{Lemma}

\theoremstyle{definition}

\newtheorem{remark}[theorem]{Remark}

\theoremstyle{plain}

\newcommand{\pa} [1]{\textcolor{black}{#1}}
\definecolor{MyDarkGreen}{rgb}{0,0.45,0}

\def\trait #1 #2 #3 {\vrule width #1pt height #2pt depth #3pt}
\def\fin{\hfill
        \trait .3 5 0
        \trait 5 .3 0
        \kern-5pt
        \trait 5 5 -4.7
        \trait 0.3 5 0
\medskip}

\newenvironment{proof}{\textit{Proof.}}{\fin}

\newcommand{\Tau}{\mathcal{T}}

\newcommand{\Eh} {\mathcal{E}_h}

\newcommand{\vrtx}{\mathsf{v}}
\newcommand{\bil}[2]{\langle#1,#2\rangle}

\newcommand{\V}{V}

\newcommand{\Vhr} [1]{V_{h,#1}}
\newcommand{\VhPr}[1]{V_{h,#1}(\P)}
\newcommand{\Vhrs}[1]{V_{h,#1}^*}

\newcommand{\Vhrp} [1]{V^{p}_{h,#1}}
\newcommand{\VhPrp}[1]{V^{p}_{h,#1}(\P)}

\newcommand{\tVhPrp}[1]{\tilde{V}^{p}_{h,#1}(\P)}

\newcommand{\Vht}[1]{\widetilde{V}_{h,#1}}
\newcommand{\VhtP}[1]{\widetilde{V}_{h,#1}(\P)}
\newcommand{\Vh} {V_{h,r}^{p}}

\newcommand{\snorm}[2]{\left|#1\right|_{#2}}
\newcommand{\norm} [2]{\left|\!\left|#1\right|\!\right|_{#2}}

\newcommand{\normV} [1]{\left|\!\left|#1\right|\!\right|_{\V}}
\newcommand{\normVP}[1]{\left|\!\left|#1\right|\!\right|_{\V,\P}}

\newcommand{\PS}[1]{\mathbbm{P}_{#1}} %% Polynomial Space (PS)
\newcommand{\hE}{h_e}
\newcommand{\hP}{h_{\P}}

\newcommand{\fh}{f_h}

\renewcommand{\P}{K}
\newcommand{\Th}{\Omega_h}

\newcommand{\SP}{S^{\P}}

\newcommand{\PiPr}[1]{\Pi^{\nabla,\P}_{#1}} %% elliptic projector, local
\newcommand{\Pizr}[1]{\Pi^{0,\P}_{#1}}  %% orthogonal projector, local

\newcommand{\EOD}{\end{document}}

\newcommand{\nx}{n_x}
\newcommand{\ny}{n_y}
\newcommand{\tx}{\tau_x}
\newcommand{\ty}{\tau_y}

\newcommand{\DIM}{\textrm{dim}}

\newcommand{\NP}{\mathcal{N}^{\P}}
\newcommand{\restrict}[2]{{#1}_{|#2}}

\newcommand{\btau}{\bm\tau}

\begin{document}

% \nofiles
\begin{frontmatter}

  \title{The conforming virtual
    element method \\ for polyharmonic problems}

  \author[MOXA]{P.~F.~Antonietti}
  \author[T5]  {, G. Manzini}
  \author[MOXV]{, and M. Verani}

  \address[MOXA]{
    MOX, Dipartimento di Matematica,
    Politecnico di Milano, Italy;
    \emph{e-mail: paola.antonietti@polimi.it}
  }
  \address[T5]{
    Group T-5,
    Theoretical Division,%
    Los Alamos National Laboratory,
    Los Alamos, NM,
    USA;
    \emph{e-mail: gmanzini@lanl.gov}
  }
  \address[MOXV]{
    MOX, Dipartimento di Matematica,
    Politecnico di Milano, Italy;
    \emph{e-mail: marco.verani@polimi.it}
  }

  % Abstract
  % ----------------------------
  \begin{abstract}
\pa{In this work, we exploit the capability of virtual element methods in accommodating approximation spaces featuring high-order continuity to numerically approximate differential problems of the form $\Delta^p u =f$, $p\ge1$. More specifically,  we develop and analyze the conforming virtual element method for the numerical approximation of polyharmonic boundary value problems, and prove an abstract result that states the convergence of the method in the energy norm.}
  \end{abstract}

\medskip


\end{frontmatter}

% Activate to display a given date or no date
\date{\today}

\maketitle

%% Introduction

\raggedbottom

\section{Introduction}
In the recent years, there has been a tremendous interest to numerical
methods that approximate partial differential equations (PDEs) on
computational meshes with arbitrarily-shaped polygonal/polyhedral
(\pa{polytopic}, for short) elements.
A nonexhaustive list of such methods include
the Mimetic Finite Difference method (see e.g., 
\cite{%HymanShashkovSteinberg_1997,%
  BreLipSha05,%
 % Gyrya-Lipnikov-Manzini:2016,%
  Brezzi-Buffa-Lipnikov:2009,%
  %BeiraodaVeiga-Lipnikov-Manzini:2011,%
 % BeiraodaVeigaManziniPutti_2015,%
  %Lipnikov-Manzini-Moulton-Shashkov:2016,%
  %Lipnikov-Manzini:2014,%
  AntoniettiBigoniVerani_2013,%
  BeiraoManziniLipnikov_2014,%
  %Lipnikov-Manzini-Shashkov:2014,%
  %Manzini-Lipnikov-Moulton-Shashkov:2017,%
  AntoniettiFormaggiaScottiVeraniVerzotti_2016}), 
the Polygonal Finite Element Method (see e.g., 
\cite{SukumarTabarraei_2004}),%
 % SukumarTabarraei_2008,%
  %Manzini-Russo-Sukumar:2014
%%
the polygonal Discontinuous Galerkin Finite Element Methods (see e.g., 
\cite{AntBreMar09,%
  Canetal14,%
  CangianiDongGeorgoulisHouston_2016,%
  AntoniettiHoustonet_al_2016,%
  %AntoniettiFacciolaRussoVerani_2016,%
  %%CangianiDongGeorgoulis_2017,%
 % AntoniettiPennesi_2017,%
 AntoniettiPennesiHouston_2018})
the Hybridizable Discontinuous Galerkin and Hybrid High-Order Methods (see e.g., 
\cite{%Cockburn-Gopalakrishnan-Lazarov:2009,%
  CockburnDongGuzman_2008,%
  DiPietroErnLemaire_2014}),
  %CockburnDiPietroErn_2016
%%
the Gradient Discretization method (see e.g., 
\cite{EymardGuichardHerbin_2012}.%
  %DroniouEymardGallouetHerbin_2013,%
  %DiPietro-Droniou-Manzini:2018
%%
%the Finite Volume Method
\cite{Droniou:2014}),
%%
%and the BEM-based FEM \cite{Weisser_basic}.  
%%
An alternative approach that \pa{is} also proved to be very successful is
provided by the Virtual Element method (VEM), which was originally
proposed in~\cite{BeiraodaVeigaBrezziCangianiManziniMariniRusso_2013} for the numerical treatment of
second-order elliptic problems~\cite{CangianiManziniRussoSukumar_2015,CangianiGyryaManziniSutton_2017}, and readily extended to
%%
%inear and nonlinear elasticity~\cite{VEMelasticity,GainTalischiPaulino_2014},
%plate bending problems~\cite{BrezziMarini_2013},
Cahn-Hilliard equation~\cite{AntoniettiBeiraoScacchiVerani_2016},
Stokes equations~\cite{AntoniettiBeiraoMoraVerani_2014},
Laplace-Beltrami equation~\cite{SVEMbasic},
Darcy-Brinkam equation~\cite{Vacca:2018},
discrete topology optimization problems~\cite{AntoniettiBruggiScacchiVerani_2017},
fracture networks problems~\cite{BenedettoBerronePieracciniScialo_2014},
eigenvalue problems~\cite{GardiniVacca_2017}. %,CertikGardiniManziniVacca_2018}.
The mixed virtual element formulation was proposed in~\cite{
  BrezziFalkMarini_2014};
  %,%BeiraoBrezziMariniRusso_2016} 
  the nonconforming Virtual element
formulations was proposed for second-order elliptic problems
in~\cite{AyusoLipnikovManzini_2016}, and later extended to general
advection-reaction-diffusion problems, Stokes equation, the biharmonic
problems, the \pa{eigenvalue} problems, and the Schrodinger equation in~\cite{%
  %BerroneBorioManzini_2018, 
  CangianiGyryaManzini_2016,%
  CangianiManziniSutton_2017,%
  ZhaoChenZhang_2016,%
  GardiniManziniVacca_2018,%
  AntoniettiManziniVerani_2018}.
%%
%The $p$- and $hp-$version of the VEM were developed in
%\cite{hpVEMbasic,hpVEMcorner} and 
Efficient multigrid methods for the
resulting linear system of equations in
\cite{AntoniettiMascottoVerani_2018}.
A posteriori error estimates can be found in
\cite{CangianiGeorgoulisPryerSutton_2017}.\\

In this work, we propose the conforming VEM for the numerical
approximation of polyharmonic problems.
A peculiar feature of VEM is the possibility of designing
approximation spaces characterized by \pa{high-order} continuity
properties~\cite{BeiraodaVeigaManzini_2014}.
This turns out to be crucial when differential operators of order
higher than two have to be considered, as, for example, in the
numerical treatment of biharmonic problems (see, e.g., the plate
bending problem or the Cahn-Hilliard equation) and polyharmonic
problems.
The numerical approximation of polyharmonic problems has been first
addressed in the eighties by~\cite{Bramble1985} and, more recently,
in~\cite{Barrett2004,Gudi2011,Wang2013,Schedensack2016,Gallistl2017}.
It is worth mentioning an increasing interest in the numerical
approximation of models involving high-order differential operators,
e.g., \cite{Lowengrub1,Lowengrub2,Miranville2015} in the context of
sixth order Cahn-Hilliard equations.
To the best of our knowledge, the conforming VEM proposed in this
article is the first work addressing the approximation of
arbitrary-order polyharmonic problems on polygonal meshes.
%%
%Moreover, when simplicial meshes are considered, this numerical method
%is new also for the finite element literature.

\medskip
The outline of the paper is as follows. 
In Section \ref{sec:polyharmonic}, we introduce the continuous
polyharmonic problem involving the differential operator $\Delta^p$
for any integer $p\geq 1$.
In Section \ref{sec:VEM}, we introduce the conforming VEM
approximation of arbitrary order.
In this case, the global VEM space is made of $C^{p-1}$ functions.
As a collateral result, we obtain a virtual element formulation that
includes the VEM for the Poisson and the biharmonic equation, where
the basis functions are globally $C^r$ for $r\geq 1$.
An abstract result proves the convergence of the method in the energy
norm that correspond to the differential operator $\Delta^p$.
In this section, we also consider an alternative formulation with
virtual element spaces of arbitrarily regular basis functions by
enriching the ``bulk'' degrees of freedom.
In Section~\ref{sec:convergence:analysis}, we derive the error
estimates in different norms.
%% 
\begin{comment}
In Section~\ref{sec:implementation}, we discuss difficulties
concerning a straightforward implementation of the method that follows
the guidelines of~\cite{hitchhikersguideVEM}.
\end{comment}
%% 
Finally, in Section~\ref{sec:conclusions}, we offer our final comments
and conclusions.

\medskip
\emph{Notation and technicalities.}  
Throughout the paper, we consider the usual multi-index notation.
In particular, if $v$ is a sufficiently regular bivariate function and
$\nu=(\nu_1,\nu_2)$ a multi-index with $\nu_1$, $\nu_2$ nonnegative
integer numbers, the function
$D^{\nu}v=\partial^{|\nu|}v\slash{\partial x_1^{\nu_1}\partial
  x_2^{\nu_2}}$ is the partial derivative of $v$ of order
$|\nu|=\nu_1+\nu_2>0$.
For $\nu=(0,0)$, we adopt the convention that $D^{\nu}v$ coincides
with $v$.
Also, for \pa{the sake of exposition}, we may use the shortcut notation
$\partial_{x}v$, $\partial_{y}v$, $\partial_{xx}v$, $\partial_{xy}v$,
$\partial_{yy}v$, to denote the first- and second-order partial
derivatives along the coordinate directions $x$ and $y$;
$\partial_{n}v$, $\partial_{t}v$, $\partial_{nn}v$, $\partial_{nt}v$,
$\partial_{tt}v$ to denote the first- and second-order normal and
tangential derivatives of order one and two along a given mesh edge;
and $\partial^{m}_{n}v$ and $\partial^{m}_{t}v$ to denote the normal
and tangential derivative of $v$ of order $m$ along a given mesh edge.
Finally, let $\mathbf{n}=(\nx,\ny)$ and $\btau=(\tx,\ty)$ be the unit 
normal and tangential \pa{vectors} to a given edge $e$ of \pa{an arbitrary polygon} $\P$.
\pa{We recall the following relations}
between the first derivatives of $v$:
\begin{align}
  \partial_nv      = \nx\partial_xv + \ny\partial_yv,\quad
  \partial_{\tau}v = \tx\partial_xv + \ty\partial_yv,
  \label{eq:edge:derivatives:first}
\end{align}
and the second derivatives of $v$:
\begin{align}
  \partial_{nn}v       = \textbf{n}^T\mathsf{H}(v)\textbf{n},\quad
  \partial_{n\tau}v    = \textbf{n}^T\mathsf{H}(v)\btau,\quad
  \partial_{\tau\tau}v = \btau^T     \mathsf{H}(v)\btau,
  \label{eq:edge:derivatives:second}
\end{align}
where matrix $\mathsf{H}(v)$ is the Hessian of $v$, i.e.,
$\mathsf{H}_{11}(v)=\partial_{xx}v$,
$\mathsf{H}_{12}(v)=\mathsf{H}_{21}(v)=\partial_{xy}v$,
$\mathsf{H}_{22}(v)=\partial_{yy}v$.

% \begin{align*}
%   \mathsf{H}(v) = 
%   \left(
%      \begin{array}{ll}
%      \partial_{xx}v & \partial_{xy}v \\[0.5em]
%      \partial_{yx}v & \partial_{yy}v
%      \end{array}
%   \right).
% \end{align*}

%\begin{align}
%  \partial_nv         &= \nx\partial_xv + \ny\partial_yv,\quad
%  \partial_{\tau}v    = \tx\partial_xv + \ty\partial_yv,\quad %\\[0.5em]
%  %%  ----
%  \partial_{nn}v      = \nx\nx\partial_{xx}v + 2\nx\ny\partial_{xy}v + \ny\ny\partial_{yy}v,\nonumber\\[0.5em]
%  \partial_{\tau\tau}v &= \tx\tx\partial_{xx}v + 2\tx\ty\partial_{xy}v + \ty\ty\partial_{yy}v,\quad
%  \partial_{n\tau}v    = \nx\tx\partial_{xx}v + 2\nx\ty\partial_{xy}v + \ny\ty\partial_{yy}v.
%  %\label{eq:edge:derivatives}
%  %% ----
%\end{align}

%!TEX root=main_vem_polyh.tex

%%%%%%%%%%%%%%%%%%%%%%%%%%%%%%%%%%%%%%%%%
\section{The continuous polyharmonic problem}
\label{sec:polyharmonic}

Let $\Omega\subset \mathbb{R}^2$ be a convex polygonal domain with
boundary $\Gamma$. 
\pa{For an integer $p\geq 1$,} we are interested in developing the conforming  \pa{Virtual Element} method
for the numerical approximation of the following problem:
\begin{subequations}\label{eq:poly:pblm:continuous}
  \begin{align}
    \Delta^p u     & = f \qquad \text{in~}\Omega,\label{eq:poly:pblm:1}\\
    \partial^j_n u & = 0 \qquad \text{for~}j=0,\ldots,p-1\text{~on~}\Gamma,\label{eq:poly:pblm:2}
  \end{align}
\end{subequations}
\pa{(recall the conventional notation $\partial^0_n u=u$)}.
Let 
\pa{
$$V \equiv H^p_{0}(\Omega)=\big\{v\in
H^p(\Omega):\partial^j_nv=0\text{~on~}\Gamma,\,j=0,\ldots,p-1\big\}.$$}
We denote the duality pairing between $V$
and its dual $V^*$ by $\langle \cdot, \cdot \rangle$.
The variational formulation of \eqref{eq:poly:pblm:continuous} reads as:
\emph{Find $u\in V$ such that}
\begin{equation}\label{eq:poly:pblm:wp}
  a(u,v) = \bil{f}{v} \qquad\forall v\in V,
\end{equation}
where, for any nonnegative integer $\ell$, the bilinear form on the
left is given by:
\begin{align}
  a(u,v) = 
  \begin{cases}
    \,\int_{\Omega} \nabla\Delta^\ell u\cdot\nabla\Delta^\ell v \,dx & \mbox{for~$p=2\ell+1$},\\[1em]
    \,\int_{\Omega} \Delta^\ell u\,\Delta^\ell v \,dx                & \mbox{for~$p=2\ell$}.
  \end{cases}
\end{align}
Whenever $f\in L^2(\Omega)$, we may consider the duality pairing
between $L^2(\Omega)$ and itself given by the $L^{2}$-inner product:
\begin{align}
  \bil{f}{v} = (f,v) = \int_{\Omega} f v dx.
 \label{eq:poly:pblm:p-rhs}
\end{align}

The existence and uniqueness of the solution to
\eqref{eq:poly:pblm:wp} follows from the Lax-Milgram
lemma%~\cite{Lax-Milgram:1953} 
because of the continuity and coercivity
of the bilinear form $a(\cdot,\cdot)$ with respect to
$\|\cdot\|_V:=\vert \cdot \vert_{p,\Omega}$ which is a norm on
$H^p_0(\Omega)$.
Moreover, since $\Omega$ is a convex polygon, from~\cite{Gazzola-book}
we know that $u\in H^{2p-m}(\Omega)\cap H^{p}_{0}(\Omega)$ if $f\in
H^{-m}(\Omega)$, $m\leq p$ and it holds that $\norm{u}{2p-m}\leq
C\norm{f}{-m}$.
In the following, we denote the coercivity and continuity constants of
$a(\cdot,\cdot)$ by $\alpha$ and $M$, respectively.

\section{The conforming  \pa{Virtual Element} approximation}
\label{sec:VEM}
\subsection{Abstract framework}
\label{subsec:abstract:framework}
Let $\big\{\Th\big\}_{h}$ be a sequence of decompositions of $\Omega$ where each mesh $\Th$ is a collection of nonoverlapping
polygonal elements $\P$ with boundary $\partial\P$, and  \pa{let} $\Eh$ be the set
of edges $e$ of $\Th$.
Each mesh is labeled by $h$, the diameter of the mesh, defined as
usual by $h=\max_{\P\in\Th} h_{\P}$, where
$h_{\P}=\sup_{\mathbf{x},\mathbf{y}\in\P}\vert\mathbf{x}-\mathbf{y}\vert$. 
We denote the set of vertices in $\Tau_h$ by $\V_h=\V_h^i \cup
\V_h^{\Gamma}$, where $\V_h^i$ and $\V_h^{\Gamma}$ are the subsets of
interior and boundary vertices, respectively.
Accordingly, $\V_h^K$ is the set of vertices of $K$.
The symbol $h_{\vrtx}$ denotes the average of the diameters of the
polygons sharing the vertex $\vrtx$. \pa{For functions
in $\Pi_{\P\in\Th}H^{p}(\P)$, we define
the seminorm $\norm{v}{h}^2=\sum_{\P\in\Th}a^{\P}(v,v)$, being $a^{\P}(\cdot,\cdot)$ the restriction of $a(\cdot,\cdot)$ to $\P$.}\\

The formulation of the  \pa{Virtual Element} method for solving
problem~\eqref{eq:poly:pblm:wp} only requires three mathematical
objects: \pa{the finite dimensional conforming \pa{Virtual Element} space $\Vhrp{r}\subset V$}, the bilinear form
$\pa{a_h(\cdot,\cdot)}$, and the linear functional $\langle f_h, \cdot \rangle$.
Their definition is the topic of this section.
Using such objects, we formulate the VEM as: \emph{Find
  $u_h\in\Vhrp{r}$ such that}
\begin{align}
  a_h(u_h,v_h) = \bil{f_h}{v_h}
  \quad\forall v_h\in\Vhrp{r}.
  \label{eq:poly:VEM}
\end{align}
The well-posedness of the VEM given in~\eqref{eq:poly:VEM}, which
implies existence and uniqueness of the solution $u_h$, is a
consequence of the Lax-Milgram lemma.
An abstract convergence result is available, which depends only on the
following assumptions:
\begin{description}
\item [\textbf{(H1)}] for each $h$ and an assigned integer number
  $r\geq 2p-1$ we are given:
  \begin{enumerate}
\medskip
  \item the \emph{global}  \pa{Virtual Element} space $\Vhrp{r}$ with the
    following properties:
    \begin{description}
    \item[-] $\Vhrp{r}$ is a finite dimensional subspace of
      $H^{p}_{0}(\Omega)$; 
    \item[-] its restriction $\VhPrp{r}$ to any element $\P$ of a given
      mesh $\Th$, \pa{called} the \emph{local}  \pa{Virtual Element} space, is a
      finite dimensional subspace of $H^{p}(\P)$;
    \item[-] \pa{$\PS{r}(\P) \subset \VhPrp{r}$ where $\PS{r}(\P)$ is the space polynomials of
      degree up to \pa{$r\geq1$} defined on $\P$}
    \end{description}
    
    \medskip
  \item the symmetric and coercive bilinear form
    $a_h:\Vhrp{r}\times\Vhrp{r}\to\mathbbm{R}$ admitting the decomposition
    \begin{align*}
      a_h(u_h,v_h) = \sum_{\P\in\Th}a^{\P}_h(u_h,v_h)
      \quad\forall u_h,\,v_h\in\Vhrp{r},
    \end{align*}
    where each local summation term $a^{\P}_h(\cdot,\cdot)$ is also a
    symmetric and coercive bilinear form;
    
    \medskip
  \item an element $f_h$ of the dual space $V^*_h$, which allows us
    to define the continuous linear functional $\bil{f_h}{\cdot}$.
  \end{enumerate}

  \medskip
\item [\textbf{(H2)}] for each $h$ and \pa{each} mesh element $\P\in\Th$, the
  local symmetric bilinear form $\pa{a^{\P}_{h}(\cdot, \cdot)}$ possesses the two
  following properties:
  \begin{description}
  \item[$(i)$] $r$-\textbf{Consistency}: for every polynomial
    $q\in\PS{r}(\P)$ and function $\Vhrp{r}(\P)$ we have:
    \begin{align}
      a^{\P}_h(v_h,q) = a^{\P}(v_h,q);
      \label{eq:poly:r-consistency}
    \end{align}
  \item[$(ii)$] \textbf{Stability}: there exist two positive constants
    $\alpha_*$, $\alpha^*$ independent of $h$ and $\P$ such that for
    every $v_h\in\VhPrp{r}$ it holds:
    \begin{align}
      \alpha_*a^{\P}(v_h,v_h)\leq a^{\P}_h(v_h,v_h)\leq \alpha^*a^{\P}(v_h,v_h).
      \label{eq:poly:stability}
    \end{align}
  \end{description}

\end{description}
To apply the Lax-Milgram lemma we need $\pa{a_{h}(\cdot, \cdot)}$ to be coercive and
continue.
The coercivity of $\pa{a_{h}(\cdot, \cdot)}$ follows from the coercivity of $\pa{a(\cdot, \cdot)}$ and the
stability property \textbf{(H2)} (with coercivity constant
$\alpha^*\alpha$).
The continuity of $\pa{a_{h}(\cdot, \cdot)}$ follows from its symmetry, assumption
\textbf{(H2)} and the continuity of $\pa{a(\cdot, \cdot)}$ (with continuity constant
$\alpha^*M$).
\pa{Denoting by $\PS{r}(\Th)$ the space of piecewise (possibly discontinuous) polynomials defined over the mesh $\Th$, the following abstract convergence result hold.}

\medskip
\begin{theorem}
  \label{theorem:poly:abstract:energy:norm}
  Let $u$ be the solution of the variational
  problem~\eqref{eq:poly:pblm:wp}.
  Then, for every  \pa{Virtual Element} approximation $u_I$ in $\Vhrp{r}$
  and any piecewise polynomial approximation $u_{\pi}\in\PS{r}(\Th)$
  of $u$ we have:
  \begin{align}
    \normV{u-u_h}\leq C
    \Big(
    \normV{u-u_I} + \norm{u-u_{\pi}}{h} + \norm{\fh-f}{\Vhrs{r}}
    \Big),
    \label{eq:poly:abstract:energy:norm}
  \end{align}
  where $C$ is a constant independent of $h$ that may depend on
  $\alpha$, $\alpha_*$, $\alpha^*$, $M$, \pa{and $r$,} and,
  \begin{align}
    \norm{f-\fh}{\Vhrs{r}}
    = \sup_{v_h\in\Vhrp{r}\backslash{\{0\}}}\frac{\bil{f-\fh}{v_h}}{\normV{v_h}}
  \end{align}
  is the approximation error of the right-hand side given in the norm
  of the dual space $\Vhrs{r}$.
\end{theorem}
\begin{proof}
  The proof of this theorem is similar to the proofs of the
  convergence theorem in the energy norm for the  \pa{Virtual Element}
  approximation of lower-order elliptic
  problems~\cite{BeiraodaVeigaBrezziCangianiManziniMariniRusso_2013,BrezziMarini_2013}.
  We briefly sketch how the proof works for completeness of
  exposition.
  First, an application of the triangular inequality implies that:
  \begin{align}
    \normV{u-u_h}\leq\normV{u-u_I}+\normV{u_I-u_h}.
    \label{eq:poly:abstract:proof:00}
  \end{align}
  Let $\delta_h=u_h-u_I$.
  Starting from the definition of $\normV{\,\cdot\,}$, we find that:
  \begin{align*}
    \begin{array}{rll}
      \alpha_*\normV{\delta_h}^2   
      &= \alpha_*a(\delta_h,\delta_h)                                                                              &\quad \mbox{\big[use~\eqref{eq:poly:stability}\big]}\nonumber\\[0.5em]
      &\quad\leq a_h(\delta_h,\delta_h)                                                                            &\quad \mbox{\big[use~$\delta_h=u_h-u_I$\big]}\nonumber\\[0.5em]
      &\quad\leq a_h(\delta_h,u_h)-a_h(\delta_h,u_I)                                                                &\quad \mbox{\big[use~\eqref{eq:poly:VEM}\big]}\nonumber\\[0.5em]
      &\quad\leq \bil{\fh}{\delta_h}-\sum_{\P\in\Th}a^{\P}_h(\delta_h,u_I)                                           &\quad \mbox{[add $\pm u_{\pi}$\big]}\nonumber\\[0.5em]
      &\quad\leq \bil{\fh}{\delta_h}-\sum_{\P\in\Th}\Big(a^{\P}_h(\delta_h,u_I-u_{\pi})+a^{\P}_h(\delta_h,u_{\pi})\Big) &\quad \mbox{\big[use~\eqref{eq:poly:r-consistency}\big]}\nonumber\\[0.5em]
      &\quad\leq \bil{\fh}{\delta_h}-\sum_{\P\in\Th}\Big(a^{\P}_h(\delta_h,u_I-u_{\pi})+a^{\P}  (\delta_h,u_{\pi})\Big) &\quad \mbox{[add $\pm u$\big]}\nonumber\\[0.5em]
      &\quad\leq \bil{\fh}{\delta_h}-\sum_{\P\in\Th}\Big(a^{\P}_h(\delta_h,u_I-u_{\pi})+a^{\P}  (\delta_h,u_{\pi}-u) + a^{\P}(\delta_h,u)\Big) &\quad \mbox{\big[use~\eqref{eq:poly:pblm:wp}\big]}\nonumber\\[0.5em]
      &\quad=    \bil{\fh-f}{\delta_h}-\sum_{\P\in\Th}\Big(a^{\P}_h(\delta_h,u_I-u_{\pi})+a^{\P}  (\delta_h,u_{\pi}-u)\Big).
    \end{array}
  \end{align*}
  Then, we use~\eqref{eq:poly:stability}, add and subtract $u$, use
  the continuity of $a^{\P}$, sum over all the elements $\P$, divide
  by $\normV{\delta_h}$, take the supremum of the right-hand side
  error term on $\Vhrp{r}\backslash{\{0\}}$, and obtain
  \begin{align}
     \alpha_*\normV{\delta_h} \leq
     \sup_{ v_h\in\Vhrp{r}\backslash{ \{0\} } } \frac{\vert\bil{\fh-f}{v_h}\vert}{\normV{v_h}} 
     + M\left( \alpha^*\normV{u_I-u} + (1+\alpha^*)\norm{u-u_{\pi}}{h} \right).
     \label{eq:poly:abstract:proof:15}
  \end{align}
  The assertion of the theorem follows by
  using~\eqref{eq:poly:abstract:proof:15}
  in~\eqref{eq:poly:abstract:proof:00} and suitably defining constant \pa{the}
  $C$.
\end{proof}

%\subsection{Preliminaries}
%% 
Let $\P\subset\mathbb{R}^2$ be a polygonal element and set
\begin{align*}
  a^{\P}(u,v) = 
  \begin{cases}
    \,\int_{\P} \nabla\Delta^\ell u\cdot\nabla\Delta^\ell v \,dx & \mbox{for~$p=2\ell+1$},\\[1em]
    \,\int_{\P} \Delta^\ell u\,\Delta^\ell v \,dx                & \mbox{for~$p=2\ell$}.
  \end{cases}
\end{align*}
For an odd $p$, i.e., $p=2\ell+1$, a repeated application of the
integration by parts formula yields
\begin{align} 
  a^{\P}(u,v) 
  =& -\int_{\P} \Delta^p u\,v\,dx + \int_{\partial\P}\partial_n(\Delta^\ell u)\,\Delta^\ell v\,ds\nonumber \\[0.5em]
  &  +\sum_{i=1}^\ell
  \left( 
    \int_{\partial\P}\partial_n(\Delta^{p-i} u)\,\Delta^{i-1}v\,ds
    -\int_{\partial\P}\Delta^{p-i}u\,\partial_n(\Delta^{i-1}v)\,ds
  \right),
  \label{eq:poly:intbyparts:odd:p}
\end{align}
while, for an even $p$, i.e., $p=2\ell$, we have
\begin{align} 
  a^{\P}(u,v) 
  = -\int_{\P}\Delta^p u\,v\,dx
  + \sum_{i=1}^\ell
  \left(
    \int_{\partial\P}  \partial_n(\Delta^{p-i} u)\,\Delta^{i-1}v\,ds
    -\int_{\partial\P} \Delta^{p-i} u\,\partial_n(\Delta^{i-1}v)\,ds
  \right).
  \label{eq:poly:intbyparts:even:p}
\end{align} 

\subsection{Virtual element spaces}
For $p\geq 1$ and $r\geq 2p-1$, the local  \pa{Virtual Element} space on
element $\P$ is defined by
\begin{align}
  \VhPrp{r} = \Big\{v_h\in H^p(\P)\,:\,
  & \Delta^p v_h\in\PS{r-2p}(\P),\,D^{\nu}v_h\in C^0(\partial\P), \vert\nu\vert\leq p-1, \nonumber\\
  & v_h\in\PS{r}(e),\,\partial^i_nv_h\in\PS{r-i}(e),\,i=1,\ldots,p-1~\forall e\in\partial\P
  \Big\},
\end{align}
with the conventional notation that $\PS{-1}(\P)=\{0\}$.
The  \pa{Virtual Element} space $\VhPrp{r}$ contains the space of
polynomials $\PS{r}(\P)$, for $r\geq 2p-1$. 
Moreover, for $p=1$, it coincides with the conforming  \pa{Virtual Element}
space for the Poisson equation
\cite{BeiraodaVeigaBrezziCangianiManziniMariniRusso_2013}; for $p=2$, it
coincides with the conforming  \pa{Virtual Element}
space for the  biharmonic equation \cite{BrezziMarini_2013}.\\

We characterize the \pa{functions} in $\VhPrp{r}$ through \pa{the following}
\emph{degrees of freedom}:
\medskip
\begin{description}
\item[$(D1)$] $h_{\vrtx}^{|\nu|} D^{\nu}v_h(\vrtx)$, $\vert\nu\vert\leq p-1$ for any
  vertex $\vrtx$ of $\P$;  
\medskip
\item[$(D2)$] $\displaystyle\hE^{-1}\int_eqv_h\,ds$ for any
  $q\in\PS{r-2p}(e)$ and \pa{any} edge $e$ of $\partial\P$;
  \medskip
\item[$(D3)$] $\displaystyle \hE^{-1+j}\int_eq\partial^j_nv_h\,ds$ for
  any $q\in\PS{r-2p+j}(e)$, $j=1,\ldots,p-1$ and \pa{any} edge $e$ of
  $\partial\P$;
  \medskip
\item[$(D4)$] $\displaystyle\hP^{-2}\int_{\P}qv_h\,ds$ for any
  $q\in\PS{r-2p}(\P)$.\\
\end{description}
Here, as usual, we assume that $\PS{-n}(\pa{\cdot})=\{0\}$ for $n\geq1$.
In $(D3)$, \pa{the} index $j$ starts from $1$ instead of $0$ since for $j=0$ we
would find the degrees of freedom that are already listed in $(D2)$
(recall that $\partial^j_nv_h=v_h$ for $j=0$).
%\footnote{\MarcoM{Check! Non
 %   dovrebbe essere $p>m/2$? { MV: a me sembra sia corretto cosi' come nel testo.}}}
%%%%%%%%%%%%%%%%%%%%%%%%%%%%%%%%%%%%%%%%%%%%%%%%%%%%%%%%%%%%%%%%%%%%%%%%%
We note that for any sufficiently regular two-dimensional domain
$\Omega$ we have the embedding 
{\color{black}$C^m({\Omega})\subset H^p(\Omega)$ if $m\leq p-1$}.
This regularity is reflected by the previous choice of the degrees of
freedom, which allows us to reconstruct the trace of $v_h$ and the
derivatives $\partial^j_nv_h\in\PS{r-j}(e)$ on each edge of
$\partial\P$.
Since these polynomial traces on a given edge only depend on the edge
degrees of freedom, the traces are the same from inside the two mesh
elements sharing that edge.
To interpolate $v_h\in\PS{r}(e)$ we need $r+1$ conditions for each
edge $e$.
Let $\vrtx_{A}$ and $\vrtx_{B}$ denote the vertices of edge $e$ and
use the shortcut notation: $v_{A}=v_h(\vrtx_{A})$,
$\partial_{n}v_{A}=\partial_{n}v_h(\vrtx_{A})$, etc.
Then,
\begin{itemize}
  % \medskip
\item the degrees of freedom $(D1)$ provides $v_{A}$, $v_{B}$,
  $\partial^{k}_{\tau}v_A$ and $\partial^{k}_{\tau}v_B$ for
  $k=1,\ldots,p-1$, i.e., $2p$ degrees of freedom, which are enough to
  interpolate $\restrict{v_h}{e}$ in $\PS{r}(e)$ if $r=2p-1$.
  When $r>2p-1$ the remaining $(r+1)-2p$ conditions required to
  interpolate $\restrict{v_h}{e}$ in $\PS{r}(e)$ are provided by the
  degrees of freedom $(D2)$.
  The tangential derivatives $\restrict{\partial^{k}_{\tau}v_h}{e}$ in
  $\PS{r-k}(e)$ for $k=1,\ldots,r-1$ can be obtained by deriving $k$
  times the interpolated polynomial $\restrict{v_h}{e}$ along $e$;

  % \medskip
\item similarly, for each $j=1,\ldots,p-1$, the degrees of freedom
  $(D1)$ provides $2(p-j)$ conditions, i.e.,
  $\partial^{k}_{\tau}\partial^{j}_{n}v_A$ and
  $\partial^{k}_{\tau}\partial^{j}_{n}v_B$, for $k=j,\ldots,p-1$.
  The remaining $r+1-2(p-j)$ conditions to interpolate
  $\partial^{j}_{n}v_h$ in $\PS{r-j}(e)$ are provided by the
  $(r-2p+j)+1$ degrees of freedom $(D3)$.
  The tangential derivatives
  $\restrict{\partial^{k}_{\tau}\partial^{j}_{n}v_h}{e}$ in
  $\PS{r-j-k}(e)$ for $k=1,\ldots,r-j-1$ can be obtained by deriving
  $k$ times the interpolated polynomial $\partial^{j}_{n}v_h$ along
  $e$.
\end{itemize}
\medskip
\noindent
Figure~\ref{fig:trih:dofs} illustrates the degrees of freedom on a
given edge $e$ for $p=3$ (triharmonic case) and $r=5,6$.
\begin{figure}
  \begin{center}
    \begin{tabular}{cc}
      \includegraphics[scale=0.5]{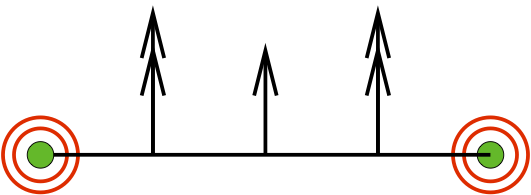} & \hspace{2cm}
      \includegraphics[scale=0.5]{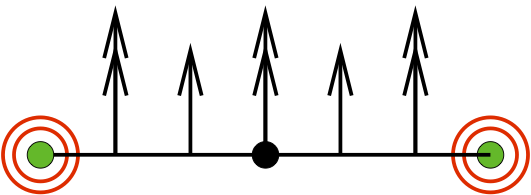} \\
      $p=3,\,r=5$ & \hspace{2cm} $p=3,\,r=6$
    \end{tabular}
    \caption{Triharmonic problem: edge degrees of freedom of the
    \pa{Virtual Element} space $\VhPr{r}$.
      Here, $p$ is the order of the partial differential operator and
      $p=3$ corresponds to the triharmonic case; $r=5,6$ are the
      integer parameters that specify the maximum degree of the
      polynomial subspace $\PS{r}(\P)$ of the VEM space
      $\VhPr{r}$.
      The (green) dots at the vertices represent the vertex values
      and each (red) vertex circle represents an order of derivation.
      The (black) dot on the edge represents the moment of
      $\restrict{v_h}{e}$; the arrows represent the moments of
      $\restrict{\partial_nv_h}{e}$; the double arrows represent the
      moments of $\restrict{\partial_{nn}v_h}{e}$.
    }
     \label{fig:trih:dofs}
  \end{center}
\end{figure}
Finally, we note that the internal degrees of freedom $(D4)$ make it
possible to define the orthogonal polynomial projection of $v_h$ onto
the space of polynomial of degree $r-2p$.\\

% ---------------------------------------------------------------------------------------
The dimension of $\VhPrp{r}$ is 
\begin{align*}
  \DIM\VhPrp{r} 
  &= \frac{p(p+1)}{2}\NP + \NP\sum_{j=0}^{p-1}\PS{r-2p+j}(e) + \DIM\PS{r-2p}(\P) \\[0.5em]
  &= \frac{p(p+1)}{2}\NP + \NP\sum_{j=0}^{p-1}( r-2p+j+1)    + \frac{(r-2p+1)(r-2p+2)}{2},
\end{align*}
where $\NP$ is the number of vertices, which equals the number of
edges, of $\P$.

\medskip 
\pa{The following lemma ensures that the above choice of degrees
of freedom is unisolvent in $\VhPrp{r}$}.
\begin{lemma}\label{lemma:unisolvence}
  The degrees of freedom (D1)-(D4) are unisolvent for $\VhPrp{r}$.
\end{lemma}
\begin{proof}
  To ease the presentation, we first consider the lowest order space
  ($r=5$) for the triharmonic problem ($p=3$).
  A counting argument implies that the cardinality of the set of
  degrees of freedom $(D1)-(D3)$ is equal to the dimension of
  $V^3_{h,5}$ (note that in this specific case $(D4)$ is empty as
  there is no volumetric integral in the right-hand side
  of~\eqref{eq:poly:intbyparts:odd:p}).
  Then, we are left to prove that a function $v_h$ in $V^3_{h,5}$ is
  zero if its degrees of freedom are zero.
  From the previous discussion on the degrees of freedom, we know that
  the edge polynomial interpolation of the traces of $v_h$,
  $\partial_{n}v_h$, $\partial_{nn}v_h$ and $\partial_{\tau\tau}v_h$,
  and, hence, $\Delta v_h=\partial_{tt}v_h+\partial_{nn}v_h$, must be
  zero if the degrees of freedom of $v_h$ are zero.
  Hence, from \eqref{eq:poly:intbyparts:odd:p} with $\ell=1$, we find that
  $\|\nabla(\Delta v_h)\|^2_{L^2(\P)} = 0$ inside $\P$, which implies
  that $\Delta v_h$ is constant in $\P$.
  Using the Divergence Theorem we find that:
  \begin{align*}
    \Delta v_h|\P| 
    = \int_{\P}\Delta v_h\,dx
    = \int_{\partial\P}\partial_n v_h\,ds
    = 0.
  \end{align*}
  Therefore, $v_h$ is the solution of the boundary value problem
  $\Delta v_h=0$ in $\P$ with boundary conditions $v_h=0$ on
  $\partial\P$, and, thus, $v_h=0$ in $\P$.

  The case of a generic $p$ can be treated analogously by properly
  employing relations \eqref{eq:poly:intbyparts:odd:p} (odd $p$) and
  \eqref{eq:poly:intbyparts:even:p} (even $p$) in combination with
  the \pa{following observations}:
  \begin{description}
  \item[$(a)$] the polynomial trace
    $\restrict{\Delta^{\nu}v_h}{e}=\partial_n^\alpha\partial_\tau^\beta
    v_h$ for every integers $\alpha$, $\beta$, and $\nu$ (with
    $\nu\geq1$) such that $\alpha+\beta=2\nu$ must be zero if the
    degrees of freedom $(D1)-(D3)$ of $v_h$ are zero;
  \item[$(b)$] the volumetric integrals in
    \eqref{eq:poly:intbyparts:odd:p} and
    \eqref{eq:poly:intbyparts:even:p} are zero if the degrees of
    freedom $(D4)$ are zero;
  \item[$(c)$] $a^K(v_h,v_h)$ is a norm on $H^p_0(K)$.
  \end{description}
\end{proof}

%%%%%%%%%%%%%%%%%%%%%%%%%%%%%%%%%%%%%%%%%%%%%%%%%%%%%%%%%%%%%%%%%%%%%%%%%

%% contruction of the elliptic projection
\medskip
To define the elliptic projection $\PiPr{r}:\VhPrp{r}\to\PS{r}(\P)$,
we first need to introduce the \emph{vertex average projector}
$\widehat{\Pi}^{\P}:\VhPrp{r}\to\PS{0}(\P)$, which projects any smooth
enough function defined on $\P$ onto the space of constant
polynomials.
Let $\psi$ be a continuous function defined on $\P$.
The \emph{vertex average projection} of $\psi$ onto the constant
polynomial space is defined as:
\begin{align}
  \widehat{\Pi}^{\P}\psi = \frac{1}{\NP}\sum_{\vrtx\in\partial\P}\psi(\mathbf{x}_{\vrtx}),
  \label{eq:trih:vertex:average:projection}
\end{align}
where $\mathbf{x}_{\vrtx}$ is the position of vertex $\vrtx$.
Finally, we define the elliptic projection
$\PiPr{r}:\VhPrp{r}\to\PS{r}(\P)$ as the solution of the finite
dimensional variational problem
\begin{align}
  a^{\P}(\PiPr{r}v_h,q)               &= a^{\P}(v_h,q)\phantom{ \widehat{\Pi}^{\P}D^{\nu}v_h } \forall q\in\PS{r}(\P),\label{eq:poly:Pi:A}\\[0.5em]
  \widehat{\Pi}^{\P}D^{\nu}\PiPr{r}v_h &= \widehat{\Pi}^{\P}D^{\nu}v_h\phantom{ a^{\P}(v_h,q) } |\nu|\leq { p-1}.         \label{eq:poly:Pi:B}
\end{align}
Such operator has two important properties:

\medskip
\begin{description}
\item[$(i)$] it is a polynomial-preserving operator in the sense that
  $\PiPr{r} q=q$ for every $q\in\PS{r}(\P)$.  
  \medskip
\item[$(ii)$] $\PiPr{r} v_h$ is \emph{computable} using only the
  degrees of freedom of $v_h$.
  In fact, in view of the integration by parts
  formulas~\eqref{eq:poly:intbyparts:odd:p} and
  \eqref{eq:poly:intbyparts:even:p}, the right-hand side
  of~\eqref{eq:poly:Pi:A} takes the form (depending on the parity of
  $p$):
  \begin{align}
    a^{\P}(v_h,q) = & -\int_{\P} \Delta^p q\,v_h\,dx + \int_{\partial\P}\partial_n(\Delta^\ell q)\,\Delta^\ell v_h\,ds\nonumber \\[0.5em]
  &  +\sum_{i=1}^\ell
  \left\{ 
    \int_{\partial\P}\partial_n(\Delta^{p-i} q)\,\Delta^{i-1}v_h\,ds
    -\int_{\partial\P}\Delta^{p-i}q\,\partial_n(\Delta^{i-1}v_h)\,ds
  \right\}\pa{,}
    \label{eq:poly:odd:computability}
  \end{align}
  or 
  \begin{align}
    a^{\P}(v_h,q) = -\int_{\P}\Delta^p q\,v_h\,dx
    + \sum_{i=1}^\ell
    \left\{ 
      \int_{\partial\P}  \partial_n(\Delta^{p-i} q)\,\Delta^{i-1}v_h\,ds
      -\int_{\partial\P} \Delta^{p-i} q\,\partial_n(\Delta^{i-1}v_h)\,ds
    \right\}.
    \label{eq:poly:even:computability}
  \end{align}
 %%%%%%%%%% 
  In~\eqref{eq:poly:odd:computability} and 
  \eqref{eq:poly:even:computability}, $\Delta^{p-i} q$, and
  $\Delta^pq$ are easily computable from $q$.
  The volumetric integral  on $\P$ can be expressed using the degrees of
  freedom $(D5)$ since it is the moment of $v_h$ against $\Delta^p q$,
  which is a polynomial of degree $r-2p$.
  The edge traces of $\Delta^\ell v_h$, $\partial_n (\Delta^{i-1} v_h)$ and $\Delta^{i-1}
  v_h$ are computable from the
  degrees of freedom $(D1)-(D4)$ of $v_h$ by solving suitable
  polynomial interpolation problems. 
\end{description}

%-----------Global space 
\medskip
Building upon the local spaces $\VhPrp{r}$ for all $\P\in\Th$, the global
conforming  \pa{Virtual Element} space $\Vhrp{r}$ is defined on $\Omega$ as
\begin{align}
  \Vhrp{r} = \Big\{ 
  v_h\in H^{p}_{{0}}(\Omega)\,:\,\restrict{v_h}{\P}\in\VhPrp{r}
  \,\,\forall\P\in\Th
  \Big\}. 
  \label{eq:poly:global:space}
\end{align}
\pa{We remark that} the associated global space is made of $C^{p-1}$ functions.
Indeed, the restriction of a  \pa{Virtual Element} function $v_h$ to each
element $\P$ belongs to $H^p(\P)$ and glues  with $C^{p-1}$-regularity
across the internal mesh faces. The global degrees of freedom induced by the local degrees of freedom
are listed as follows:
  \medskip
\begin{description}
\item[-] $h_{\vrtx}^{|\nu|} D^{\nu}v_h(\vrtx)$, $\vert\nu\vert\leq p-1$
  for every interior vertex $\vrtx$ of $\Th$;
\medskip
\item[-] $\displaystyle\hE^{-1}\int_eqv_h\,ds$ for any $q\in\PS{r-2p}(e)$ and every interior edge $e\in\Eh$;
\medskip
\item[-] $\displaystyle\int_eq\partial_nv_h\,ds$ for every
  $q\in\PS{r-5}(e)$ and every interior edge $e\in\Eh$;
\medskip
\item[-]  $\displaystyle \hE^{-1+j}\int_eq\partial^j_nv_h\,ds$ for any
  $q\in\PS{r-2p+j}(e)$ $i=1,\ldots,p-1$ and every interior edge $e\in\Eh$;
\medskip
\item[-] $\displaystyle\hP^{-2}\int_{\P}qv_h\,ds$ for any
  $q\in\PS{r-2p}(\P)$ and every $\P\in\Th$.
\end{description}
\begin{remark}
For $p=3$ (triharmonic case) we can also consider the following modified lowest order space
  \begin{align*}
    \VhtP{5} = \Big\{ 
    v_h\in H^3(\P)
    \,:\,
    &\Delta^3 v_h = 0, 
    v_h,\,\partial_nv_h,\,\partial_{nn}v_h\in C^0(\partial\P),
    \\
    &
    v_h\in\PS{5}(e),\,
    \partial_nv_h\in\PS{3}(e),\,
    \partial_{nn}v_h\in\PS{2}(e)~\forall e\in\partial\P
    \Big\}
  \end{align*}
  with associated dofs
  \begin{description}
  \item[$(D1')$] $h_{\vrtx}^{|\nu|}D^{\nu}v_h(\vrtx)$, $\vert\nu\vert\leq 2$ for any
    vertex $\vrtx$ of $\partial\P$;
    
    \medskip
  \item[$(D2')$] $\displaystyle h_{e}\int_e\partial_{nn}v_h\,ds$ for any
    edge $e$ of $\partial\P$.
    
  \end{description}
  Using the same argument of the proof of
  Lemma~\ref{lemma:unisolvence}, we can still prove that $(i)$ the
  degrees of freedom $(D1')$ and $(D2')$ are unisolvent in $\VhtP{5}$;
  $(ii)$ the space of polynomials of degree {$4$} are a subspace of $\VhtP{5}$; $(iii)$ the
  elliptic projection of $v_h$ is still computable from this choice of
  degrees of freedom; $(iv)$ the associated global space 
  \begin{align}
    \Vht{5} = \Big\{ 
    v_h\in H^{3}_0(\Omega)\,:\,v_h\vert_{\P}\in\VhtP{5}
    \,\,\forall\P\in\Th
    \Big\},
    \label{eq:trih:global:space:modified:VEM}
  \end{align}
  which is obtained by gluing together all the elemental spaces
  $\VhtP{5}$, is still made of $C^2$ functions. 
{\color{black}  Analogously, in the general case one can build the following modified lowest order spaces (containing the space of polynomials of degree {$2p-2$}  )
\begin{align}
  \tVhPrp{2p-1} = \Big\{v_h\in H^p(\P)\,:\,
  & \Delta^p v_h=0,\,D^{\nu}v_h\in C^0(\partial\P), \vert\nu\vert\leq p-1, \nonumber\\
  & v_h\in\PS{2p-1}(e),\,\partial^i_nv_h\in\PS{2p-2-i}(e),\,i=1,\ldots,p-1~\forall e\in\partial\P
  \Big\},
\end{align}
with associated dofs

\begin{description}
  \item[$(D1')$] $h_{\vrtx}^{|\nu|}D^{\nu}v_h(\vrtx)$, $\vert\nu\vert\leq p-1$ for any
    vertex $\vrtx$ of $\partial\P$;
    
    \medskip
  \item[$(D2')$] $\displaystyle h_{e}^{-1+j}\int_eq \partial_{n}^i v_h\,ds$ for any $q \in\PS{i-2}(e)$ and any 
    edge $e$ of $\partial\P$, $i=1,\ldots,p-1$.
    
  \end{description}
}
\end{remark}

\subsection{Construction of the bilinear form}
\label{subseq:VEM:bilinear:form}

%% VEM bilinear form
We \pa{write} the symmetric bilinear form
$a_{h}:\Vhrp{r}\times\Vhrp{r}\to\mathbbm{R}$ as the
sum of local terms
\begin{align}
  a_{h}(u_h,v_h) = \sum_{\P\in\Th}a_{h}^{\P}(u_h,v_h),
\end{align}
where each local term $a_h^{\P}:\VhPrp{r}\times\VhPrp{r}\to\mathbbm{R}$
is a symmetric bilinear form.
We set
\begin{align}
  a_h^{\P}(u_h,v_h) 
  = a^{\P}(\PiPr{r} u_h,\PiPr{r} v_h) 
  + \SP(u_h-\PiPr{r} u_h,v_h-\PiPr{r} v_h),
  \label{eq:poly:ah:def}
\end{align}
where $\SP:\VhPrp{r}\times\VhPrp{r}\to\mathbbm{R}$ is a symmetric
positive definite bilinear form such that
\begin{align}
  \sigma_*a^{\P}(v_h,v_h)\leq\SP(v_h,v_h)\leq\sigma^*a^{\P}(v_h,v_h)
  \qquad\forall v_h\in\VhPrp{r}\textrm{~with~}\PiPr{r} v_h=0,
  \label{eq:poly:S:stability:property}
\end{align}
for two some positive constants $\sigma_*$, $\sigma^*$ independent of
$h$ and $\P$.
The bilinear form $a_h^{\P}\pa{(\cdot,\cdot)}$ has the two fundamental properties of
\emph{consistency} and \emph{stability} stated by the following lemma.
\begin{lemma}
  The bilinear form $a_h^{\P}\pa{(\cdot,\cdot)}$ defined in~\eqref{eq:poly:ah:def}
  possesses both $(i)$ $r$-stability and $(ii)$ consistency properties
  stated in \eqref{eq:poly:r-consistency}
  and~\eqref{eq:poly:stability}, respectively, as required by
  assumption~\textbf{(H2)}.
\end{lemma}
\begin{proof}
The $r$-consistency property follows by noting that the stability term in~\eqref{eq:poly:ah:def} is zero when one of its
  entries is a polynomial of degree $r$ as $\PiPr{r}$ is a
  polynomial-preserving operator. The stability property is easily established by
  applying~\eqref{eq:poly:stability} to
  definition~\eqref{eq:poly:ah:def} and setting
  $\alpha_*=\min(\sigma_*,1)$ and $\alpha^*=\max(\sigma^*,1)$, where
  $\sigma_*$ and $\sigma^*$ are the constants defined in
  \eqref{eq:poly:S:stability:property}.
\end{proof}
%%%%%

Furthermore, $a_h^{\P}\pa{(\cdot,\cdot)}$ is $V$-elliptic and continuous for every $\P$,
and so is the global bilinear form $a_h\pa{(\cdot,\cdot)}$.
The $V$-ellipticity of $a_h^{\P}\pa{(\cdot,\cdot)}$ is indeed a consequence of the left
inequality in~\eqref{eq:poly:stability}.
Since $a_h^{\P}\pa{(\cdot,\cdot)}$ is symmetric and coercive, it is a scalar product on
$\VhPr{r}$ and satisfies the Cauchy-Schwarz inequality.
Using the right inequality in~\eqref{eq:poly:stability} we (easily)
prove the continuity of $a_h^{\P}\pa{(\cdot,\cdot)}$ with respect to norm
$\normVP{\,\cdot\,}$:
\begin{align}
  a_h^{\P}(u_h,v_h)
  &\leq \big( a_h^{\P}(u_h,u_h) \big)^{\half}\,\big( a_h^{\P}(v_h,v_h) \big)^{\half}      %\nonumber\\[0.3em]
  \leq \alpha^* \big( a^{\P}(u_h,u_h) \big)^{\half}\,\big( a^{\P}(v_h,v_h) \big)^{\half} \nonumber\\[0.5em]
  &\leq \alpha^* M\,\normVP{u_h}\,\normVP{v_h}
  \qquad\forall u_h,\,v_h\in\VhPrp{r}.
\end{align}
%%%%%%
Collecting together the local terms, we can formulate the global
$V$-ellipticity and continuity properties as follows:
\begin{align}
  \alpha_* a(v_h,v_h)\leq a_{h}(v_h,v_h)\leq \alpha^* a^{\P}(v_h,v_h)
  \qquad\forall v_h\in\Vh\label{eq:poly:stability:global}\\[0.5em]
  a_h(u_h,v_h) \leq \alpha^* M\,\normV{u_h}\,\normVP{v_h}
  \qquad\forall u_h,\,v_h\in\Vh.
  \label{eq:poly:continuity:global}
\end{align}

\subsection{Construction of the load term}
\label{subseq:VEM:load:term}
We denote by $f_h$ the piecewise polynomial approximation of $f$ on
$\Th$ given by
\begin{equation}\label{eq:rhs}
  \restrict{f_h}{\P}=\Pizr{r-p}f\pa{,}
\end{equation}
for $r\geq 2p-1$ and $\P\in\Th$. 
Then, we set
\begin{equation}\label{vem:rhs}
  \bil{f_h}{v_h} = \sum_{\P\in\Th}\int_{\P}f_h v_h\,dx.
\end{equation}
Using the definition of the $L^2$-orthogonal projection we find that
\begin{equation}\label{aux:1.1}
  \bil{f_h}{v_h} 
  = \sum_{\P\in\Tau_h}\int_{\P}\Pizr{r-p}f\,v_h\,dx
  = \sum_{\P\in\Tau_h}\int_{\P}\Pizr{r-p}\,f\Pizr{r-p}v_h\,dx
  = \sum_{\P\in\Tau_h}\int_{\P}f\,\Pizr{r-p} v_h\,dx. 
\end{equation}
The right-hand side of \eqref{aux:1.1} is computable by using the
degrees of freedom $(D1)-(D5)$ and the enhanced
approach~\cite{AhmadAlsaediBrezziMariniRusso_2013}.

\medskip
\begin{remark}
  An alternative formulation that does not require the enhancement is
  given by taking $\overline{r}=\max(0,r-2p)$ for $r\geq 2p-1$ and
  $f_h=\Pizr{\overline{r}}f$.
  The resulting approximation is suboptimal.
\end{remark}

\subsection{VEM approximation of polyharmonic problems with basis functions of arbitrary degree of continuity}
\label{sec:polyharmonic-2}

In this last section we briefly sketch the construction of global  \pa{Virtual Element} spaces with higher order of continuity. More precisely, let us consider the local  \pa{Virtual Element} space defined as before,
for $r \geq 2p-1$:
\begin{align}
  \VhPrp{r} = \Big\{ v_h\in { H^p}(\P):\, 
  & \Delta^p v_h\in\PS{r-2p}(\P), D^{\nu} v_h\in C^0(\partial\P), \vert\nu\vert\leq p-1, \nonumber\\[0.5em]
  & v_h\in\PS{r}(e),\,\partial^j_n v_h\in\PS{r-j}(e),\,j=1,\ldots,p-1~\forall e\in\partial\P\big\}.
\end{align}

Differently from the previous section,  {\color{black} we make the degrees of freedom depend on a  given parameter $t$ with $0\leq t \leq p-1$. For a given} value of $t$ we choose the \emph{degrees of
  freedom} of $V_{h}^{\P,r}$ as follows

\smallskip
\begin{description}
\item[$(D1)$] $h^{|\nu|}D^{\nu}v_h(\vrtx)$, $\vert\nu\vert\leq p-1$ for
  any vertex $\vrtx$ of $\P$;
  
{\color{black}
\medskip
\item[$(D2)$] $\displaystyle h_{e}^{-1}\int_e{ v_h q}\,ds$ for any
  $q\in\PS{r-2p}(e)$,\, for any edge $e$ of
  $\partial\P$;  
  
  \medskip
\item[$(D3)$] $\displaystyle h_{e}^{-1+j}\int_e{\partial^j_n v_h q}\,ds$ for any
  $q\in\PS{r-2p+j}(e)$,\,$j=1,\ldots,p-1$ for any edge $e$ of
  $\partial\P$;
  }
  \medskip
\item[$(D4')$] $\displaystyle h_{K}^{-2}\int_{\P} qv_h\,ds$ for any $q\in
  \PS{r-2(p-t)}(\P)$;
\end{description}
where as usual we assume $\PS{-n}(\cdot)=\{0\}$ for $n=1,
2,3,\ldots$.\\

%Note that in two dimensions $H^p(\P)\subset %C^m(\overline{\Omega})$ if
%$m\leq p-1$.
%%
In view of the above choice of the degrees of freedom, the following properties hold true:
\begin{enumerate}
\item the dofs are unisolvent. Indeed, proceeding as before,
  it is enough to use \eqref{eq:poly:intbyparts:odd:p} or
  \eqref{eq:poly:intbyparts:even:p} and observe that $\Delta^i
  v_h\vert_e = \partial_n^\alpha \partial_\tau^\beta v_h$ with
  $\alpha+\beta=2i$ is a polynomial uniquely identified by the values
  of the dofs;
\item $\PS{r}(\P)\subset V_{h}^{\P,r}$, for $r\geq 2p-1$;
\item the choice $(D4')$ instead of $(D4)$ still guarantees that the
  associated global space is made of $C^{p-1}$ functions, but now
  $(D1)$-$(D4')$ can be employed to solve a differential problem
  involving the $\Delta^{p-t}$ operator by employing $C^{p-1}(\Omega)$ basis functions.
  For instance:
  {\color{black}
  \begin{enumerate}
  \item Choosing $p$ and $t$ such that $p-t=1$ we obtain $C^{p-1}$ conforming VEM for the solution of
    the Laplacian problem. For example, for $p=3$, $t=2$ and $r=5$, the local space $V^3_{h,5}(K)$ endowed with the corresponding degrees of freedom $(D1)-(D4')$ can be employed to build a global space made of $C^2$ functions for the approximation of the Laplace problem. It is worth mentioning that the new choice $(D4')$ (differently from the original choice $(D4)$) is essential for the computability of the elliptic projection (see \eqref{eq:poly:Pi:A}-\eqref{eq:poly:Pi:B}) with respect to the bilinear form $a^K(\cdot,\cdot)=\int_K \nabla(\cdot)\nabla(\cdot) dx$.
  \item Choosing $p$ and $t$ such that $p-t=2$  we get $C^{p-1}$ conforming VEM for the solution of
    the Bilaplacian problem. For example, for $p=3$, $t=1$
   and $r=5$, similarly to the previous case, the space $V^3_{h,5}(K)$ together with $(D1)-(D4')$ gives rise to a global space made of $C^2$ functions that can be employed for the solution of the biharmonic problem. Again, the particular choice $(D4')$ makes possible the computability of the ellliptic projection with respect to the 
   bilinear form $a^K(\cdot,\cdot)=\int_K \Delta(\cdot)\Delta(\cdot) dx$.
  \end{enumerate}
  }
\end{enumerate}

\section{Convergence analysis}
\label{sec:convergence:analysis}

\subsection{Mesh regularity and polynomial interpolation error estimates}

We consider the following mesh regularity assumptions:

\medskip

\begin{description}
\item[\textbf{(M)}] There exists a positive constant $\gamma$ independent
  of $h$ (and $\P$) such that $\{\Th\}$:
  \begin{enumerate}
  \item[$(i)$] $\P$ is star-shaped with respect to every point of a
    ball of radius $\gamma\hP$, where $\hP$ is the diameter of $\P$;
  \item[$(ii)$] for every edge $e$ of the cell boundary $\partial\P$
    of every cell $\P$ of $\Th$, it holds that $\hE\leq\gamma\hP$,
    where $\hE$ denotes the length of $e$.
  \end{enumerate}
  We refer to $\gamma$ as the \emph{mesh regularity constant}.
\end{description}

\medskip
\noindent
In view of assumptions $\mathbf{M}(i)$-$\mathbf{M}(ii)$ on $\Th$, we
define, for every smooth enough function $w$ the  \pa{Virtual Element}
interpolant $w^I$, which is the function in $\Vhr{r}$ uniquely
identified by the same degrees of freedom of $w$.
More precisely, if $\chi_i(w)$ denotes the $i$-th global degree of
freedom of $w$, there exists a unique  \pa{Virtual Element} function
$w^I\in\Vhr{r}$ such that $\chi_i(w-w^I)=0$.
Combining the Bramble-Hilbert Lemma and scaling arguments as in the
finite element framework (see, e.g.,
\cite{BeiraodaVeigaBrezziCangianiManziniMariniRusso_2013}%,MoraRiveraRodriguez_2015}
and \cite{Brenner-Scott:94}) we can prove that for every $\P\in\Th$
and every function $w\in H^{\beta}(\P)$, it holds
\begin{align}\label{eq:poly:interp-VEM}
  \|w-w^I\|_{s,\P}\leq C h_{\P}^{{\min(\beta,r+1)} - s} |w|_{\beta,{\P}} 
  \qquad s=0,1,\ldots,p\quad 
  %p\leq\beta\leq {r}+1,
\end{align}
for some positive constant $C$ independent of $h$.

%%
%Here, $\PS{\ell}(\Th)$ denotes the space of the piecewise polynomials
%of maximum degree $\ell$ defined on mesh $\Th$, so that if
%$q\in\PS{\ell}(\Th)$, than its restriction $\restrict{q}{\P}$ to any
%mesh element $\P$ is a polynomial of degree $\ell$ on $\P$.
Under the same assumptions and using similar techniques, we can also prove that the existence of a piecewise polynomial approximation
$w_{\pi}\in\PS{\ell}(\Th)$ such that the local estimate
holds
\begin{align}\label{eq:poly:local:polynomial:approximation}
  \|w-w_{\pi}\|_{s,\P}\leq C h_{\P}^{{\min(\beta,\ell+1)}  - s} |w|_{\beta,{\P}} 
  \qquad s=0,1,\ldots,p,\quad 1\leq\beta\leq\ell+1,
\end{align}
for some positive constant $C$ independent of $h$ and every mesh
element $\P$.
The elliptic projection $\PiPr{\ell}w$ and the $L^2$-ortogonal
projection $\Pizr{\ell}w$ of $w$ are both instances of $w_{\pi}$ for
which estimate~\eqref{eq:poly:local:polynomial:approximation} holds.

\subsection{Convergence in the energy norm}

By using standard estimates for the interpolation error, we can derive
the convergence rate of the approximation error in the energy norm.
First, we need a technical lemma that estimates the approximation
error of the load term.

\medskip
\begin{lemma}
  \label{lemma:poly:rhs:error}
  Consider a function $f\in H^{r-(p-1)}(\Omega)$ and its $L^2$-orthogonal
  projection on the space of polynomials of degree $r-p$, denoted by
  $f_h=\Pizr{r-p}$.
  Then, there exists a positive constant $C$, which is independent of
  $h$, such that
  \begin{equation}
    \bil{f-f_h}{v_h} 
    \leq C h^{r+1}\snorm{f}{r-(p-1)}\,\snorm{v_h}{p}\qquad \forall v_h\in \Vhrp{r}.
    \label{eq:poly:rhs:error}
  \end{equation}
\end{lemma}
\begin{proof}
  First, we note that $\big(I-\Pizr{r-p}\big)f$ is orthogonal to the
  polynomials of degree $r-p$ (recall that $r\geq 2p-1$) and that $v_h$
  belongs to $\pa{H^p_0(\Omega)}$.
  We employ the Cauchy-Schwarz inequality (twice) and
  use~\eqref{eq:poly:interp-VEM} (with $w=f$ and $w_{\pi}=\Pizr{r-p}f$)
  to obtain the estimate:
  \begin{align*}
    \bil{f-f_h}{v_h} 
    &=     \sum_{\P\in\Th}\int_{\P}\big(I-\Pizr{r-p}\big)f\,\,\big(I-\Pizr{p-1}\big)v_h\,dx
    \leq   \sum_{\P\in\Th}\norm{\big(I-\Pizr{r-p}\big)f}{0,\P}\,\norm{\big(I-\Pizr{p-1}\big)v_h}{0,\P} \\[0.5em]
    &\leq C\sum_{\P\in\Th}h_{\P}^{r-(p-1)}\snorm{f}{r-(p-1),\P}\,\,h_{\P}^{p}\snorm{v_h}{p,\P}
    \leq C h^{r+1}\snorm{f}{r-(p-1)}\snorm{v_h}{p},
  \end{align*}
  where $C$ denotes a positive constant independent of $h$.
\end{proof}

\begin{theorem}
  \label{theorem:poly:energy:convg:rate}
  Let $f\in H^{r-p+1}(\Omega)$ and \pa{let} $u$ be the solution of the variational problem~\eqref{eq:poly:pblm:wp} and $u_h\in\Vhrp{r}$ \pa{be} the solution of the  \pa{Virtual Element} problem~\eqref{eq:poly:VEM}. Under the mesh regularity assumption \textbf{(M)}, we find that
  \begin{align}
    \normV{u-u_h}
    \leq C h^{r-(p-1)}\big( \snorm{u}{{ r+p+1}} + \snorm{f}{{ r-p+1}} \big).
    \label{eq:poly:energy:convg:rate}
  \end{align}
\end{theorem}
\begin{proof}
  The assertion of the theorem follows by estimating each term of the
  right-hand side of~\eqref{eq:poly:abstract:energy:norm}
  separately and using interpolation
  estimate~\eqref{eq:poly:interp-VEM} and
  Lemma~\ref{lemma:poly:rhs:error},
 \pa{cf. inequality~\eqref{eq:poly:rhs:error}}.
\end{proof}

\subsection{Convergence in lower order norms}
We first prove a technical lemma that will be useful in the error
analysis of the next subsections.
\begin{lemma}
  \label{lemma:poly:aux:pblm:psi:u-uh}
  Let $f\in H^{r-p+1}(\Omega)$ and \pa{let}  $u$ be the solution of the variational problem~\eqref{eq:poly:pblm:wp} and $u_h\in\Vhrp{r}$ \pa{be}  the solution of the  \pa{Virtual Element} problem~\eqref{eq:poly:VEM}.
  Then, for any function $\psi\in H^{\beta}(\Omega)\cap
  H^{p}_{0}(\Omega)$ ({$\beta > p$}) it holds that
  \begin{align}
    a(\psi,u-u_h) 
    \leq C h^{ (r-(p-1)) + { \min(\beta,r+1)-p }}\Big(
    \snorm{u}{{ r+p+1}} + \snorm{f}{{ r-p+1}}\Big)\norm{\psi}{\beta},
    \label{eq:poly:aux:pblm:psi:u-uh}
  \end{align}
  for some positive constant $C$ independent of $h$.
\end{lemma}
\begin{proof}
  To derive~\eqref{eq:poly:aux:pblm:psi:u-uh}, we add and substract
  the  \pa{Virtual Element} interpolant of $\psi$ denoted by $\psi_I$ to the
  left-hand side of~\eqref{eq:poly:aux:pblm:psi:u-uh}, and, then,
  use~\eqref{eq:poly:pblm:continuous} and~\eqref{eq:poly:VEM},
  and obtain:
  \begin{align}
    a(\psi,u-u_h)
    &= a(u-u_h,\psi-\psi_I) + a(u-u_h,\psi_I) \nonumber\\[0.5em]
    &=  a(u-u_h,\psi-\psi_I) + \bil{f-\fh}{\psi_I} + a_h(u_h,\psi_I)-a(u_h,\psi_I) \nonumber\\[0.5em]
    &= T_1 + T_2 + T_3.
    \label{eq:poly:aux:pblm:00.1}
  \end{align}

  \medskip
  To estimate term $T_1$, we use the continuity of $a(\cdot,\cdot)$
  with respect to the norm $\normV{\,\cdot\,}=\snorm{\,\cdot\,}{p}$,
  the estimate in the energy norm~\eqref{eq:poly:energy:convg:rate},
  and interpolation error estimate~\eqref{eq:poly:interp-VEM}
  (${ s=p}$)
  \begin{align}
    T_1 
    &\leq \normV{u-u_h}\,\normV{\psi-\psi_I}
    \leq Ch^{r-(p-1)}\Big(\snorm{u}{r+{ p}+1}+\snorm{f}{r-(p-1)}\Big)\,\snorm{\psi-\psi_I}{p}\\[0.5em]
    &\leq Ch^{(r-(p-1))+({\min(\beta,r+1)}-p)}\Big(\snorm{u}{r+{ p}+1}+\snorm{f}{r-(p-1)}\Big)\snorm{\psi}{\beta}.
    \label{eq:poly:aux:pblm:T1.1}
  \end{align}

  \medskip
  To estimate term $T_2$, we first note that
  $\big(I-\Pizr{r-p}\big)f$ is orthogonal to the polynomials of global
  degree up to $r-p$ defined on $\P$ and \pa{that} $\psi_I\in H^{{ p}}(\Omega)$.
  Then, we apply the Cauchy-Schwarz inequality (twice) and obtain:
  \begin{align*}
    \bil{f-f_h}{\psi_I} 
    &=    \sum_{\P\in\Th} \int_{\P}\big(I-\Pizr{r-p}\big)f\,\,\big(I-\Pizr{p-1}\big)\psi_I \,dx
    \leq \sum_{\P\in\Th} \norm{\big(I-\Pizr{r-p}\big)f}{0,\P}\,\norm{\big(I-\Pizr{p-1}\big)\psi_I}{0,\P}\\[0.5em]
    &\leq 
    \left(\sum_{\P\in\Th}\norm{\big(I-\Pizr{r-p}\big)f}{0,\P}^2\right)^{\frac{1}{2}}\,
    \left( \sum_{\P\in\Th}\norm{\big(I-\Pizr{p-1}\big)\psi_I}{0,\P}^2\right)^{\frac{1}{2}}.
  \end{align*}
  The first term on the right is bounded by using local
  estimate~\eqref{eq:poly:local:polynomial:approximation}:
  \begin{align*}
    \norm{\big(I-\Pizr{r-p}\big)f}{0,\P}
    \leq C\hP^{r-(p-1)}\snorm{f}{r-(p-1),\P}.
  \end{align*}
  The second term on the right is transformed by applying
  estimate~\eqref{eq:poly:local:polynomial:approximation}
  (with $w=\psi_I$ and $w_{\pi}=\Pizr{p-1}\psi_I$), adding and subtracting $\psi$ and 
  applying estimate~\eqref{eq:poly:interp-VEM}:
  \begin{align*}
    \norm{(I-\Pizr{p-1})\psi_I}{0,\P}
    \leq Ch^{p}\snorm{\psi_I}{p,\P} %\\[0.5em]
    \leq Ch^{p}\big( \snorm{\psi}{p,\P} + \snorm{\psi_I-\psi}{p,\P} \big) %\\[0.5em]
    \leq Ch^{p}\big( \snorm{\psi}{p,\P} + h^{{\min(\beta,r+1)}-p}\snorm{\psi}{\beta,\P} \big).
  \end{align*}
  Collecting all the local terms, using the Cauchy-Schwarz inequality,
  and the assumption that $\beta\geq p$ (so $h^{p}\geq h^{{\min(\beta,r+1)}}$) yields:
  \begin{align}
    T_2
    \leq Ch^{r-(p-1)}\snorm{f}{r-(p-1)}\,\Big( h^{p}\snorm{\psi}{p} + h^{p{ +}({\min(\beta,r+1)}-p)}\snorm{\psi}{\beta}\Big)
    \leq Ch^{r+1}\snorm{f}{r-(p-1)}\,\snorm{\psi}{\beta}.
    \label{eq:poly:aux:pblm:T2.1}
  \end{align}
  %We note that this term is asymptotically smaller than $T_1$ and
  %$T_3$.\footnote{\MarcoM{Questo commento \`e inutile, per\'o %sembra
 %     davvero che nella stima questo termine sia trascurabile rispetto
%      agli altri due.}}

  \medskip
  To bound $T_3$, we first split it in the summation of local terms.
  Then, we use the $r$-consistency and stability property of $a_h$,
  and the continuity property of $a$ and $a_h$, and we obtain
  \begin{align}
    T_3
    &= \sum_{\P\in\Th} \Big( a_h^{\P}(u_h,\psi_I) - a^{\P}(u_h,\psi_I) \Big)\nonumber\\[0.5em]
    &= \sum_{\P\in\Th} \Big( a_h^{\P}(u_h-u_{\pi},\psi_I) - a^{\P}(u_h-u_{\pi},\psi_I) \Big)\nonumber\\[0.5em]
    &= \sum_{\P\in\Th} \Big( a_h^{\P}(u_h-u_{\pi},\psi_I-\psi_{\pi}) - a^{\P}(u_h-u_{\pi},\psi_I-\psi_{\pi}) \Big)\nonumber\\[0.5em]
    &\leq C\normV{u_h-u_{\pi}}\,\normV{\psi_I-\psi_{\pi}}.
    \label{eq:T3:bound:00.1}
  \end{align}
  Adding and subtracting $u$ and using the estimate in the energy
  norm~\eqref{eq:poly:energy:convg:rate} and the estimate for the
  polynomial
  interpolation~\eqref{eq:poly:local:polynomial:approximation}, we
  find that
  \begin{align}
    \normV{u_h-u_{\pi}}
    \leq \normV{u_h-u} + \normV{u-u_{\pi}}
    \leq 
    Ch^{r-(p-1)}.
    %\Big( \snorm{u}{r+1} + \snorm{f}{r-(p-1)} \Big).
    \label{eq:T3:bound:u.1}
  \end{align}
  Adding and subtracting $\psi$, and, then, using
  estimates~\eqref{eq:poly:interp-VEM}
  and~\eqref{eq:poly:local:polynomial:approximation} we find that
  \begin{align}
    \normV{\psi_I-\psi_{\pi}}
    \leq \normV{\psi_I-\psi}+\normV{\psi-\psi_{\pi}}  
    \leq Ch^{{\min(\beta,r+1)}-p}\norm{\psi}{\beta}.
    \label{eq:T3:bound:psi.1}
  \end{align}
  The bound on $T_3$ following by using~\eqref{eq:T3:bound:u.1}
  and~\eqref{eq:T3:bound:psi.1} in~\eqref{eq:T3:bound:00.1}:
  \begin{align}
    T_3 
    \leq Ch^{(r-(p-1))+{\min(\beta,r+1)}-p}\norm{\psi}{\beta}.
    \label{eq:poly:aux:pblm:T3.1}
  \end{align}
  
  \medskip
  \noindent
  The assertion of the lemma follows by subtituting 
  \eqref{eq:poly:aux:pblm:T1.1}, \eqref{eq:poly:aux:pblm:T2.1}, and \eqref{eq:poly:aux:pblm:T3.1},
  in~\eqref{eq:poly:aux:pblm:00.1}.
\end{proof}

\medskip
In view of this lemma, we can readily state and prove the convergence
theorems for the four possible combinations of even and odd $p$ and
even and odd norm indices.
\begin{theorem}[Even $p$, even norms]
  Let $f\in H^{r-p+1}(\Omega)$  and \pa{let}  $u$ be the solution of the variational
  problem~\eqref{eq:poly:pblm:wp} with $p=2\ell$ and $u_h$ \pa{be}  the
  solution of the  \pa{Virtual Element} method~\eqref{eq:poly:VEM}.
  Then, there exists a positive constant $C$ independent of $h$ such
  that
  \begin{align}
    \snorm{u-u_h}{2i}
    \leq C h^{ r+1-2i }\Big(\snorm{u}{r+{p}+1}+\snorm{f}{r-(p-1)}\Big),
  \end{align}
  for every integer $i=0,\ldots,\ell-1$.
\end{theorem}
\begin{proof}
  For $i=0,\ldots,\ell-1$, let $\psi\in H^{2(p-i)}(\Omega)\cap
  H^{{ p-i}}_{0}(\Omega)$ be the solution of the problem
  \begin{align}
    \Delta^{p-i}\psi = \Delta^{i}(u-u_h) \in L^2(\Omega),
    \label{eq:aux:pblm:even-p:even-norms}
  \end{align}
  with the stability property
  \begin{align}
    \norm{\psi}{2(p-i)}\leq C\snorm{u-u_h}{2i}.
    \label{eq:poly:stab:even-p:even-norms}
  \end{align}
  We use~\eqref{eq:aux:pblm:even-p:even-norms} and integrate by parts to
  obtain:
  \begin{align*}
    \snorm{u-u_h}{2i}^2 
    &= \norm{\Delta^{i}(u-u_h)}{0}^2 
    = \big( \Delta^{i}(u-u_h),\,\Delta^{i}(u-u_h) \big)
    = \big( \Delta^{p-i}\psi,\,\Delta^{i}(u-u_h) \big)\\[0.5em]
    &= \big( \Delta^{\ell}\psi,\,\Delta^{\ell}(u-u_h) \big)
    = a(\psi,u-u_h)
  \end{align*}
  { where we employed the fact that 
  $\snorm{v}{2i} = \norm{\Delta^{i}v}{0} $ for any $v\in H^p_0(\Omega)$.}
  The assertion of the theorem follows from an application of
  Lemma~\ref{lemma:poly:aux:pblm:psi:u-uh} (use $\beta=2(p-i)$ together with $r\geq 2p-1$) and
  the stability property~\eqref{eq:poly:stab:even-p:even-norms}.
  %\footnote{
   % \MarcoM{
     % $r-(p-1) + min(p,2p-2i-p) = r-p+1 + min(p,p-2i) %= r-p+1+p-2i = r+1-2i$
   % }
  %}
  %% 
\end{proof}
    
\begin{theorem}[Even $p$, odd norms]
  Let $f\in H^{r-p+1}(\Omega)$ and  \pa{let}  $u$ be the solution of the variational
  problem~\eqref{eq:poly:pblm:wp} with $p=2\ell$ and $u_h$ \pa{be}  the
  solution of the  \pa{Virtual Element} method~\eqref{eq:poly:VEM}.
  Then, there exists a positive constant $C$ independent of $h$ such
  that
  \begin{align}
    \snorm{u-u_h}{2i+1}
    \leq C h^{ (r+1)-(2i+1) }\Big(\snorm{u}{r+{ p}+1}+\snorm{f}{r-(p-1)}\Big),
  \end{align}
  for every integer $i=0,\ldots,\ell-1$.
\end{theorem}
\begin{proof}
  For $i=0,\ldots,\ell-1$, let $\psi\in H^{2(p-i)-1}(\Omega)\cap H^{{ p-i}}_{0}(\Omega)$ be the
  solution of the problem:
  \begin{align}
    -\Delta^{p-i}\psi = \Delta^{i+1}(u-u_h) \in H^{-1}(\Omega),
    \label{eq:aux:pblm:even-p:odd-norms}
  \end{align}
  with the stability property
  \begin{align}
    \norm{\psi}{2(p-i)-1}\leq C\snorm{u-u_h}{2i+1}.
    \label{eq:poly:stab:even-p:odd-norms}
  \end{align}
  We use~\eqref{eq:aux:pblm:even-p:odd-norms} and integrate by parts
  to obtain:
  \begin{align*}
    \snorm{u-u_h}{2i+1}^2 
    &= \norm{\nabla\Delta^{i}(u-u_h)}{0}^2 
    = \big( \nabla\Delta^{i}(u-u_h),\,\nabla\Delta^{i}(u-u_h) \big)\\[0.5em]
    &= \big( \nabla\Delta^{i+1}(u-u_h),\,\nabla\Delta^{i}(u-u_h) \big)
    = \big( \Delta^{p-i}\psi,\,\Delta^{i}(u-u_h) \big)
    = \big( \Delta^{\ell}\psi,\,\Delta^{\ell}(u-u_h) \big)\\[0.5em]
    &= a(\psi,u-u_h)
  \end{align*}
 { where we employed the fact that 
  $\snorm{v}{2i+1} = \norm{\nabla\Delta^{i}v}{0} $ for any $v\in H^p_0(\Omega)$.}  
  
  The assertion of the theorem follows from an application of
  Lemma~\ref{lemma:poly:aux:pblm:psi:u-uh} (use $\beta=2(p-i)$  together with $r\geq 2p-1$) and
  the stability property~\eqref{eq:poly:stab:even-p:odd-norms}.
%  \footnote{
%    \MarcoM{
%      $r-(p-1) + min(p,2p-2i-p-1) = r-p+1 + min(p,p-2i) = r-p+1+p-2i-1 = r-2i = (r+1)-(2i+1) $
%    }
%  }
\end{proof}
  
\begin{theorem}[Odd $p$, even norms]
  Let $u$ be the solution of the variational
  problem~\eqref{eq:poly:pblm:wp} and \pa{let}  $u_h$ \pa{be}  the solution of the
  \pa{Virtual Element} method~\eqref{eq:poly:VEM}.
  Then, there exists a positive constant $C$ independent of $h$ such
  that
  \begin{align}
    \snorm{u-u_h}{2i}
    \leq C h^{ (r+1)-2i }\Big(\snorm{u}{r+{ p}+1}+\snorm{f}{r-(p-1)}\Big),
  \end{align}
    for every integer $i=0,\ldots,\ell-1$.
\end{theorem}
\begin{proof}
  For $i=0,\ldots,\ell$, let $\psi\in H^{2(p-i)}(\Omega)\cap
  H^{{ p-i}}_{0}(\Omega)$ be the solution of the problem
  \begin{align}
    -\Delta^{p-i}\psi = \Delta^{i}(u-u_h) \in L^{2}(\Omega)
    \label{eq:aux:pblm:odd-p:even-norms}
  \end{align}
  with the stability property
  \begin{align}
    \norm{\psi}{2(p-i)}\leq C\snorm{u-u_h}{2i}.
    \label{eq:poly:stab:odd-p:even-norms}
  \end{align}
  We use~\eqref{eq:aux:pblm:odd-p:even-norms} and integrate by parts
  to obtain:
  \begin{align*}
    \snorm{u-u_h}{2i}^2 
    &= \norm{\Delta^{i}(u-u_h)}{0}^2 
    = \big( \Delta^{i}(u-u_h),\,\Delta^{i}(u-u_h) \big)\\[0.5em]
    &= \big( -\Delta^{p-i}\psi,\,\Delta^{i}(u-u_h) \big)
    = \big( -\Delta^{\ell+1}\psi,\,\Delta^{\ell}(u-u_h) \big)\\[0.5em]
    &= \big( \nabla\Delta^{\ell}\psi,\,\nabla\Delta^{\ell}(u-u_h) \big)
    = a(\psi,u-u_h).
  \end{align*}
  The assertion of the theorem follows from an application of
  Lemma~\ref{lemma:poly:aux:pblm:psi:u-uh} (use $\beta=2(p-i)$  together with $r\geq 2p-1$) and
  the stability property~\eqref{eq:poly:stab:odd-p:even-norms}.
\end{proof}
  
\begin{theorem}[Odd $p$, odd norms]
  Let $u$ be the solution of the variational
  problem~\eqref{eq:poly:pblm:wp} and \pa{let}  $u_h$ \pa{be}  the solution of the
  \pa{Virtual Element} method~\eqref{eq:poly:VEM}.
  Then, there exists a positive constant $C$ independent of $h$ such
  that
  \begin{align}
    \snorm{u-u_h}{2i+1}
    \leq C h^{ (r+1)-(2i+1) }\Big(\snorm{u}{r+{ p} + 1}+\snorm{f}{r-(p-1)}\Big),
  \end{align}
  for every integer $i=0,\ldots,\ell-1$.
\end{theorem}
\begin{proof}
  For $i=0,\ldots,\ell$, let $\psi\in H^{2(p-i)-1}(\Omega)\cap
  H^{p}_{0}(\Omega)$ be the solution of the problem:
  \begin{align}
    -\Delta^{p-i}\psi = \Delta^{i+1}(u-u_h) \in H^{-1}(\Omega),
    \label{eq:aux:pblm:odd-p:even-norms}
  \end{align}
  with the stability property
  \begin{align}
    \norm{\psi}{2(p-i)-1}\leq C\snorm{u-u_h}{2i+1}.
    \label{eq:poly:stab:odd-p:odd-norms}
  \end{align}
  We use~\eqref{eq:aux:pblm:odd-p:even-norms} and integrate by parts
  to obtain:
  \begin{align*}
    \snorm{u-u_h}{2i+1}^2 
    &= \norm{\nabla\Delta^{i}(u-u_h)}{0}^2 
    = \big( \nabla\Delta^{i}(u-u_h),\,\nabla\Delta^{i}(u-u_h) \big)\\[0.5em]
    &= \big( \nabla\Delta^{i+1}(u-u_h),\,\nabla\Delta^{i}(u-u_h) \big)
    = \big( -\Delta^{p-i}\psi,\,\Delta^{i}(u-u_h) \big)
    = \big( -\Delta^{\ell+1}\psi,\,\Delta^{\ell}(u-u_h) \big)\\[0.5em]
    &= \big( \nabla\Delta^{\ell}\psi,\,\nabla\Delta^{\ell}(u-u_h) \big)
    = a(\psi,u-u_h)
  \end{align*} 
where again we employed the fact that 
  $\snorm{v}{2i+1} = \norm{\nabla\Delta^{i}v}{0} $ for any $v\in H^p_0(\Omega)$.
  The assertion of the theorem follows from an application of
  Lemma~\ref{lemma:poly:aux:pblm:psi:u-uh} (use $\beta=2(p-i)-1$  together with $r\geq 2p-1$) and
  the stability property~\eqref{eq:poly:stab:odd-p:odd-norms}.
\end{proof}

\begin{comment}
\section{Implementations}
\label{sec:implementation}
\end{comment}

\section{Conclusions}
\label{sec:conclusions}
In this paper, we developed the conforming  \pa{Virtual Element}
discretization of arbitrary order for polyharmonic problems, which
requires the discretization of operator like $\Delta^pu$ for integer
$p\geq 1$.
To this end, we introduced local and global  \pa{Virtual Element}
approximation spaces together with suitable discrete bilinear forms 
for odd and even $p$. 
The convergence of the method has been proved and optimal error
estimates derived in suitable norms.
The numerical implementation of the current method \pa{deserves a careful study}
because of the severe ill-conditioning of the polyharmonic
differential operator and the need of high order polynomials, whose
degree should be at least $5$ in the simplest case $p=3$.
For this reason, the implementation of this conforming VEM \pa{is under} investigation and will be addressed in a forthcoming
publication.

\section*{Acknowledgements}
\pa{The work of the first author has been partially funded by SIR Startin grant n. RBSI14VT0S funded by MIUR. The first and last authors has been partially supported by INdAM-GNCS.}
The work of the second author was partially supported by the
Laboratory Directed Research and Development Program (LDRD),
U.S. Department of Energy Office of Science, Office of Fusion Energy
Sciences, and the DOE Office of Science Advanced Scientific Computing
Research (ASCR) Program in Applied Mathematics Research, under the
auspices of the National Nuclear Security Administration of the
U.S. Department of Energy by Los Alamos National Laboratory, operated
by Los Alamos National Security LLC under contract DE-AC52-06NA25396. This article is assigned the number LA-UR-18-29151.

%\bibliographystyle{abbrv}
%\bibliography{biblio}

\end{document}